\documentclass[11pt,reqno]{amsart}

\usepackage{amsmath,amssymb,amsfonts,amscd,amsthm}
\usepackage{mathrsfs,mathabx}
\usepackage{latexsym,graphicx,cite,cases,array,booktabs}

\newtheorem{theorem}{Theorem}[section]

\newtheorem{lemma}[theorem]{Lemma}

\newtheorem{remark}[theorem]{Remark}

\numberwithin{equation}{section}

\title[an inverse source problem]{Stability for an inverse source problem of the damped biharmonic plate equation}

\author[P. Li]{Peijun Li}
\address{Department of Mathematics, Purdue University, West Lafayette, Indiana
47907, USA}
\email{lipeijun@math.purdue.edu}

\author[X. Yao]{Xiaohua Yao}
\address{School of Mathematics and Statistics, China Central Normal University,
Wuhan, Hubei, China}
\email{yaoxiaohua@mail.ccnu.edu.cn}

\author[Y. Zhao]{Yue Zhao}
\address{School of Mathematics and Statistics, China Central Normal University,
Wuhan, Hubei, China}
\email{zhaoyueccnu@163.com}


\begin{document}

\begin{abstract}
This paper is concerned with the stability of the inverse source problem for the damped biharmonic plate equation in three dimensions. The stability estimate consists of the Lipschitz type data discrepancy and the high frequency tail of the source function, where the latter decreases as the upper bound of the frequency increases. The stability also shows exponential dependence on the constant damping coefficient. The analysis employs Carleman estimates and time decay estimates for the damped plate wave equation to obtain an exact observability bound and depends on the study of the resonance-free region and an upper bound of the resolvent of the biharmonic operator with respect to the complex wavenumber.
\end{abstract}


\maketitle

\section{Introduction}

Consider the damped biharmonic plate equation in three dimensions
\begin{align}\label{eqn}
\Delta^2u(x, k) - k^2 u(x, k) - {\rm i}k\sigma u(x, k) = f(x), \quad x\in\mathbb R^3,
\end{align}
where $k>0$ is the wavenumber, $\sigma>0$ is the damping coefficient, and $f\in L^2(\mathbb R^3)$ is a assumed to be a real-valued function with a compact support contained in $B_R=\{x\in\mathbb R^3: |x|<R\}$, where $R>0$ is a constant. Let $\partial B_R$ be the boundary of $B_R$. Since the problem is formulated in the open domain, the Sommerfeld radiation condition is imposed usually on $u$ and $\Delta u$ to ensure the well-posedness of the problem \cite{TS}. This paper is concerned with the inverse source problem of determining $f$ from the boundary measurements 
\begin{align*}
u(x, k), \, \nabla u(x, k), \, \Delta u(x, k), \, \nabla \Delta u(x, k), \quad x\in\partial B_R
\end{align*}
corresponding to the wavenumber $k$ given in a finite interval.

In general, there is no uniqueness for the inverse source problems of the wave equations at a fixed frequency \cite{BLT, LYZ}. Computationally, a more serious issue is the lack of stability, i.e., a small variation of the data might lead to a huge error in the reconstruction. Hence it is crucial to examine the stability of the inverse source problems. In \cite{BLT}, the authors initialized the study of the inverse source problem for the Helmholtz equation by using multi-frequency data. Since then, it has become an active research topic on the inverse source problems via multiple frequency data in order to overcome the non-uniqueness issue and enhance the stability. The increasing stability was investigated for the inverse source problems of various wave equations which include the acoustic, elastic, and electromagnetic wave equations \cite{BLZ, CIL, EI-18, EI-20, LY, LZZ} and the Helmholtz equation with attenuation \cite{IL}. On the other hand, it has generated sustained interest in the mathematics community on the boundary value problems for higher-order elliptic operators \cite{GGS}. The biharmonic operator, which can be encountered in models originating from elasticity for example, appears as a natural candidate for such a study \cite{RR, MMM}. Compared with the equations involving the second order differential operators, the model equations with the biharmonic operators are much less studied in the community of inverse problems. We refer to \cite{AP, Iw, KLU, LKU, TS} and the references cited therein on the recovery of the lower-order coefficients by using either the far-field pattern or the Dirichlet-to-Neumann map on the boundary. In a recent paper \cite{LYZ}, the authors demonstrated the increasing stability for the inverse source problem of the biharmonic operator with a zeroth order perturbation by using multi-frequency near-field data. The main ingredient of the analysis relies on the study of an eigenvalue problem for the biharmonic operator with the hinged boundary conditions. But the method is not applicable directly to handle the biharmonic operator with a damping coefficient.

Motivated by \cite{CIL, IL}, we use the Fourier transform in time to reduce the inverse source problem into the identification of the initial data for the initial value problem of the damped biharmonic plate wave equation by lateral Cauchy data. The Carleman estimate is utilized to obtain an exact observability bound for the source function in the framework of the initial value problem for the corresponding wave equation, which connects the scattering data and the unknown source function by taking the inverse Fourier transform. An appropriate rate of time decay for the damped plate wave equation is proved in order to justify the Fourier transform. Then applying the results in \cite{LYZ} on the resolvent of the biharmonic operator, we obtain a resonance-free region of the data with respect to the complex wavenumber and the bound of the analytic continuation of the data from the given data to the higher wavenumber data.  By studying the dependence of analytic continuation and of the exact observability bound for the damped plate wave equation on the damping coefficient, we show the exponential dependence of increasing stability on the damping constant. The stability estimate consists of the Lipschitz type of data discrepancy and the high wavenumber tail of the source function. The latter decreases as the wavenumber of the data increases, which implies that the inverse problem is more stable when the higher wavenumber data is used. But the stability deteriorates as the damping constant becomes larger.  It should be pointed out that due to the existence of the damping coefficient, we can not obtain a sectorial resonance-free region for the data as that in \cite{CIL, LY}. Instead, we choose a rectangular resonance-free region as that in \cite{LZZ}, which leads to a double logarithmic type of the high wavenumber tail for the estimate. 

This paper is organized as follows. In section \ref{resolvent analysis}, the direct source problem is discussed; the resolvent is introduced for the elliptic operator, and its resonance-free region and upper bound are obtained. Section \ref{inverse problem} is devoted to the stability analysis of the inverse source problem by using multi-frequency data. In appendix \ref{exact}, we use the Carleman estimate to derive an exact observability bound with exponential dependence on the damping coefficient. In appendix \ref{time decay}, we prove an appropriate rate of time decay for the damped plate wave equation to justify the Fourier transform.

\section{The direct source problem}\label{resolvent analysis}

In this section, we discuss the solution of the direct source problem and study the resolvent of the biharmonic operator with a damping coefficient. 

\begin{theorem}
Let $f\in L^2(\mathbb R^3)$ with a compact support. Then there exists a unique solution $u$ of Schwartz distribution to \eqref{eqn} for every $k>0$. Moreover, the solution satisfies
\[
|u(x, k)|\leq C(k, f)e^{-c(k, \sigma)|x|}
\]
as $|x|\to\infty,$ where $C(k, f)$ and $c(k, \sigma)$ are positive constants depending on $k, f$ and $k, \sigma$, respectively. 
\end{theorem}

\begin{proof}
Taking the Fourier transform of $u(x, k)$ formally with respect to the spatial variable $x$, we define
\begin{align*}
u^*(x, k) = \int_{\mathbb R^3} e^{{\rm i}x\cdot\xi} \frac{1}{|\xi|^4 - k^2 - {\rm i}k\sigma} \hat{f}(\xi) {\rm d}\xi, \quad x\in\mathbb R^3,
\end{align*}
where 
\[
\hat{f}(\xi) = \frac{1}{(2\pi)^3} \int_{\mathbb R^3} f(x) e^{-{\rm i}x\cdot\xi} {\rm d}x.
\]
It follows from the Plancherel theorem that for each $k>0$ we have that $u^*(\cdot, k)\in H^4(\mathbb R^3)$ and satisfies the equation \eqref{eqn} in the sense of Schwartz distribution. 

Denote 
\[
G(x, k) =  \int_{\mathbb R^3} e^{{\rm i}x\cdot\xi} \frac{1}{|\xi|^4 - k^2 - {\rm i}k\sigma}{\rm d}\xi.
\]
By a direct calculation we can write $u^*(x, k)$ as
\begin{align}\label{u^*}
u^*(x, k) = (G*f)(x) = \frac{1}{2\kappa^2}\int_{\mathbb R^3} \Big(\frac{e^{{\rm i}\kappa |x - y|}}{4\pi|x - y|} - \frac{e^{-\kappa |x - y|}}{4\pi|x - y|}\Big) f(y){\rm d}y,
\end{align}
where $\kappa = (k^2 + {\rm i}k\sigma)^{\frac{1}{4}}$ such that $\Re \kappa>0$ and $\Im \kappa>0$. Since $f$ has a compact support, we obtain from \eqref{u^*} that the solution $u^*(x, k)$ satisfies the estimate
\[
|u^*(x, k)|\leq C(k, f)e^{-c(k, \sigma)|x|}
\]
as $|x|\to\infty$, where $C(k, f)$ and $c(k, \sigma)$ are positive constants depending on $k, f$ and $k, \sigma$, respectively. By direct calculations, we may also show that $\nabla u^*$ and $\Delta u^*$ have similar exponential decay estimates.  

Next is show the uniqueness. Let $\tilde{u}^*(x, k)$ be another Schwartz distributional solution to \eqref{eqn}. Clearly we have
\begin{align*}
(\Delta^2 - k^2 - {\rm i}k\sigma) (u^* - \tilde{u}^*) = 0.
\end{align*}
Taking the Fourier transform on both sides of the above equation yields
\[
(|\xi|^4 - k^2 - {\rm i}k\sigma) ( \widehat{u^* - \tilde{u}^*}) (\xi) = 0.
\]
Notice that for $k>0$ we have $|\xi|^4 - k^2 - {\rm i}k\sigma\neq 0$ for all $\xi\in\mathbb R^3$. Taking the generalized inverse Fourier transform gives $u^* - \tilde{u}^* = 0$, which proves the uniqueness.
\end{proof}

To study the resolvent we let 
\[
u^*(x, \kappa) := u(x, k), \quad \kappa = (k^2 + {\rm i}k\sigma)^{\frac{1}{4}},
\]
where $\Re \kappa>0$ and $\Im\kappa>0$. By \eqref{eqn}, $u^*$ satisfies 
\begin{align*}
\Delta^2 u^* - \kappa^4 u^*= f.
\end{align*}

Denote by $\mathcal{R} = \{z\in\mathbb C: (\delta, +\infty)\times (-d, d)\}$ the infinite rectangular slab, where $\delta$ is any positive constant and $d\ll 1$. For $k\in\mathcal{R}$, denote the resolvent 
\[
R(k) := (\Delta^2 - k^2 - {\rm i}k\sigma)^{-1}.
\]
Then we have $R(\kappa) = (\Delta^2 - \kappa^4)^{-1}$. Hereafter, the notation $a\lesssim b$ stands for $a\leq Cb,$ where $C>0$ is a generic constant which may change step by step in the proofs.

\begin{lemma}\label{direct}
For each $k\in\mathcal{R}$ and $\rho\in C_0^\infty(B_R)$ the resolvent operator $R(k)$ is analytic and has the following estimate:
\begin{align*}
\|\rho R(k) \rho\|_{L^2(B_R)\rightarrow H^j(B_R)} \lesssim |k|^{\frac{j}{2}} e^{2R(\sigma + 1)|k|^{\frac{1}{2}}}, \quad j = 0, 1, 2, 3, 4.
\end{align*}
\end{lemma}

\begin{proof}
It is clear to note that for a sufficiently small $d$, the set $\{(k^2 + {\rm i}k\sigma)^{\frac{1}{4}}: k\in\mathcal{R}\}$ belongs to the first quadrant. Consequently, $(k^2 + {\rm i}k\sigma)^{\frac{1}{4}}$ is analytic with respect to $k\in\mathcal{R}$. By \cite[Theorem 2.1]{LYZ}, the resolvent $\mathcal{R}(\kappa)$ is analytic in $\mathbb{C}\backslash\{0\}$ and the following estimate holds:
\begin{align}\label{free}
\|\rho R(\kappa) \rho\|_{L^2(B_R)\rightarrow H^j(B_R)}\lesssim |\kappa|^{-2}\langle\kappa\rangle^j (e^{2R (\Im\kappa)_-} + e^{2R (\Re\kappa)_-}), \quad j=0, 1, 2, 3, 4,
\end{align}
where $x_{-}:=\max\{-x,0\}$ and  $\langle\kappa\rangle = (1 + |\kappa|^2)^{1/2}$. On the other hand, letting $k = k_1 + {\rm i}k_2$, we have from a direct calculation that 
\[
k^2 + {\rm i}k\sigma = k_1^2 - k_2^2 - k_2\sigma + (2k_1k_2 + k_1\sigma){\rm i}.
\]
It is easy to see that if $d$ is sufficiently small, which gives that $|k_2|$ is sufficiently small, there is a positive lower bound for $|k^2 + {\rm i}k\sigma|$ with $k\in\mathcal{R}$ and then $|\kappa|>c$ for some positive constant $c$. The proof is completed by replacing $\kappa$ with $(k^2 + {\rm i}k\sigma)^{\frac{1}{4}}$ in \eqref{free}. 
\end{proof}

\section{The inverse source problem}\label{inverse problem}

In this section, we address the inverse source problem of the damped biharmonic plate equation and present an increasing stability estimate by using multi-frequency scattering data. 

Denote
\begin{align*}
\|u(x, k)\|^2_{\partial B_R}&:= \int_{\partial B_R}\Big( ( k^4 + k^2)|u(x, k)|^2 + k^2|\nabla u(x, k)|^2 \\
&\quad+ (k^2 + 1)|\Delta u(x, k)|^2 + |\nabla\Delta u(x, k)|^2\Big) {\rm d}s(x).
\end{align*}
The following lemma provides a relation between the unknown source function and the boundary measurements. Hereafter, by Remark \ref{rd}, we assume that $f\in H^n(B_R)$ where $n\geq 4.$

\begin{lemma}\label{control}
Let $u$ be the solution to the direct scattering problem \eqref{eqn}. Then 
\begin{align*}
\|f\|_{L^2(B_R)}^2&\lesssim 2e^{C\sigma^2} \int_{0}^{+\infty} \|u(x, k)\|^2_{\partial B_R} {\rm d}k.
\end{align*}
\end{lemma}

\begin{proof}
Consider the initial value problem for the damped biharmonic plate wave equation 
\begin{align}\label{Kirchhoff}
\begin{cases}
\partial_t^2 U(x, t) + \Delta^2 U(x, t) + \sigma\partial_t U(x, t) = 0, &\quad (x, t)\in B_R\times (0, +\infty),\\
U(x, 0) = 0, \quad \partial_tU(x, 0) = f(x), &\quad x\in B_R.
\end{cases}
\end{align}
We define $U(x, t) = 0$ when $t<0$  and denote $U_T(x, t) = U(x, t)\chi_{[0, T]}(t)$ and 
 \[
 \widehat{U_T} (x, k) = \int_0^T U(x, t) e^{{\rm i}kt} {\rm d}t.
 \] 
 By the decay estimate \eqref{de} we have that $U(x, t)\in L_t^2(0, +\infty)$ and $\lim_{T\rightarrow\infty}U_T(x, t) = U(x, t)$ in $L^2_t(\mathbb R)$ uniformly for all $x\in\mathbb R^3$. It follows from the Plancherel Theorem that $\widehat{U_T}$ also converges  in $L^2_k(\mathbb R)$ to a function $u_*(x, k)\in L^2_k(\mathbb R)$ uniformly for all $x\in\mathbb R^3$, which implies that $u_*(x, k)$ is the Fourier transform of $U(x, t)$.
  
Denote by $\langle \cdot, \cdot \rangle$ and $\mathcal{S}$ the usual scalar inner product of $L^2(\mathbb R^3)$ and the space of Schwartz functions, respectively. We take $u_*(x, k)$ as a Schwartz distribution such that $u_*(x, k)(\varphi) = \langle u_*, \varphi \rangle$ for each $\varphi\in\mathcal{S}$. In what follows, we show that $u_* (x, k)$ satisfies the equation \eqref{eqn} in the sense of Schwartz distribution.  

First we multiply both sides of the wave equation \eqref{Kirchhoff} by a Schwartz function $\varphi$ and take integration over $\mathbb R^3$. Using the wave equation \eqref{Kirchhoff} and the integration by parts with respect to the $t$ variable over $[0, T]$ for $T>0$, we obtain
\begin{align}\label{identity}
0 &= \int_0^{T} \langle \partial_t^2 U + \Delta^2 U + \sigma\partial_t U, \varphi \rangle e^{{\rm i}kt} {\rm d}t\notag\\
&= e^{{\rm i}kT}\langle \partial_t U(x, T), \varphi \rangle - {\rm i}k e^{{\rm i}kT}\langle  U(x, T), \varphi \rangle + \sigma e^{{\rm i}kT}\langle U(x, T),\varphi \rangle\notag\\
&\quad - \langle \partial_t U(x, 0), \varphi \rangle +\Big\langle\int_0^T (\Delta^2 U - k^2 U - {\rm i}k\sigma U)  e^{{\rm i}kt} {\rm d}t, \varphi \Big\rangle.
\end{align} 
It follows from the decay estimate \eqref{de} that $|\partial_t U(x, t)|, |U(x, t)|\lesssim {(1 + t)^{-\frac{3}{4}}}$ uniformly for all $x\in\mathbb R^3$, which give 
\[
\lim_{T\rightarrow \infty} e^{{\rm i}kT}\langle \partial_t U(x, T), \varphi \rangle = \lim_{T\rightarrow \infty} {\rm i}k e^{{\rm i}kT}\langle  U(x, T), \varphi \rangle = \lim_{T\rightarrow \infty}  \sigma e^{{\rm i}kT}\langle U(x, T),\varphi \rangle= 0.
\]
On the other hand, we have from the integration by parts that 
\begin{align}\label{identity_2}
&\Big\langle\int_0^T (\Delta^2 U - k^2 U - {\rm i}k\sigma U)  e^{{\rm i}kt} {\rm d}t, \varphi \Big\rangle\notag\\
&=  \Big\langle\int_0^T U {\rm d}t, \Delta^2\varphi\Big\rangle + \Big\langle\int_0^T (- k^2 U - {\rm i}k\sigma U)  e^{{\rm i}kt} {\rm d}t, \varphi \Big\rangle.
\end{align}
Since $\lim_{T\rightarrow+\infty} \widehat{U_{T}} (x, k) = u_*(x, k)$ in $L^2_k(\mathbb R)$ uniformly for $x\in\mathbb R^3$, we can choose a positive sequence $\{T_n\}_{n=1}^\infty$ such that 
$\lim_{n\rightarrow\infty} T_n= +\infty$ and
 $\lim_{n\rightarrow\infty} \widehat{U_{T_n}} (x, k) = u_*(x, k)$ pointwisely for a.e. $k\in\mathbb R$ and uniformly for all $x\in\mathbb R^3$. Define a sequence of Schwartz distributions $\{\mathcal{D}_n\}_{n=1}^\infty\subset\mathcal{S}^\prime$ as follows
\[
\mathcal{D}_n(\varphi) := \langle \widehat{U_{T_n}}, \varphi \rangle, \quad \varphi\in\mathcal{S}.
\]
Since $\lim_{n\rightarrow\infty} \widehat{U_{T_n}} (x, k) = u_*(x, k)$ for a.e. $k\in\mathbb R$ and uniformly for all $x\in\mathbb R^3$, we have
\[
\lim_{n\rightarrow\infty} \mathcal{D}_n(\varphi) = \langle u_*, \varphi \rangle.
\]
Consequently, replacing $T$ by $T_n$ in \eqref{identity_2} and letting $n\rightarrow\infty$, we get
\begin{align*}
&\lim_{n\rightarrow\infty} \Big(\big\langle\int_0^{T_n} U {\rm d}t, \Delta^2\varphi\big\rangle + \big\langle\int_0^{T_n} (- k^2 U - {\rm i}k\sigma U)  e^{{\rm i}kt} {\rm d}t, \varphi \big\rangle\Big)\\
&= u_*(\Delta^2\varphi) - k^2 u_*(\varphi) - {\rm i}k\sigma u_*(\varphi)\\
&= (\Delta^2 - k^2 - {\rm i}k\sigma)u_*(\varphi),
\end{align*}
which further implies by \eqref{identity} that
\[
(\Delta^2 - k^2 - {\rm i}k\sigma)u_*(\varphi) = \langle f, \varphi\rangle
\]
for every $\varphi\in\mathcal{S}$. Then $u_* (x, k)$ is a solution to the equation \eqref{eqn} as a Schwartz distribution. Furthermore, it follows from the uniqueness of the direct problem that we obtain $u_* (x, k) = u(x, k)$, which gives that $u(x, k)$ is the Fourier transform of $U(x, t)$. 
 
 By Theorem \ref{thm-decay}, we have the estimates
 \begin{align*}
|\partial_t^2 U|,  \,\, |\partial_t U|, \,\, |\partial_t\nabla U|, \,\, |\partial_t\Delta U|,  \,\, |\Delta U|, \,\, |\nabla\Delta U| \lesssim (1 + t)^{-\frac{3}{4}}.
 \end{align*}
Moreover, they are continuous and belong to $L^2_t(\mathbb R)$ uniformly for all $x\in\mathbb R^3$. Similarly, we may show that 
 \begin{align*}
 &\widehat{\partial_t^2 U} = -k^2 u, \,\,  \widehat{\partial_t U} = {\rm i}k u, \,\,  \widehat{\partial_t\nabla U} = {\rm i}k\nabla u,\\
 &\widehat{\partial_t \Delta U} = {\rm i}k\Delta u, \,\,  \widehat{\Delta U} = \Delta u, \,\, \widehat{\nabla\Delta U} = \nabla\Delta u.
 \end{align*}
It follows from Plancherel's theorem that 
 \begin{align}\label{identities}
 &\int_0^{+\infty} \Big(|\partial_t^2 U|^2 +  |\partial_t U|^2 + |\partial_t\nabla U|^2 + |\partial_t\Delta U|^2 +  |\Delta U|^2 + |\nabla\Delta U|^2\Big){\rm d}t\notag\\
 &= \int_{-\infty}^{+\infty} \Big(|k^2 u|^2 + |k u|^2 + |k\nabla u|^2 + |k \Delta u|^2 + |\Delta u|^2 + |\nabla\Delta u|^2 \Big){\rm d}k.
 \end{align}
 By \eqref{identities} and the exact observability bounds \eqref{control_1}, we obtain
\begin{align*}
\|f\|^2_{L^2(B_R)}&\lesssim  e^{C\sigma^2}\int_{-\infty}^{+\infty} \int_{\partial B_R} \Big(( k^4 + k^2)|u(x, k)|^2 + k^2|\nabla u(x, k)|^2 \\
&\quad+ (k^2 + 1)|\Delta u(x, k)|^2 + |\nabla\Delta u(x, k)|^2 \Big){\rm d}s(x){\rm d}k\\
&\lesssim e^{C\sigma^2} \int_{-\infty}^\infty \|u(x, k)\|^2_{\partial B_R} {\rm d}k.
\end{align*}
Since $f(x)$ is real-valued, we have $ \overline{u(x, k)} = u(x, -k)$ for $k\in\mathbb R$ and then 
\[
\int_{-\infty}^\infty \|u(x, k)\|^2_{\partial B_R} {\rm d}k = 2\int_{0}^\infty \|u(x, k)\|^2_{\partial B_R} {\rm d}k,
\]
which completes the proof.
\end{proof}

Let $\delta$ be a positive constant and define
\begin{align*}
I(k) &= \int_\delta^k \|u(x, \omega)\|^2_{\partial B_R} {\rm d}s(x){\rm d}\omega. 
\end{align*}
The following lemma  gives a link between the values of an analytical function for small and large arguments (cf. \cite[Lemma A.1]{LZZ}).

\begin{lemma}\label{ac}
Let $p(z)$ be analytic in the infinite rectangular slab
\[
R = \{z\in \mathbb C: (\delta, +\infty)\times (-d, d) \}, 
\]
where $\delta$ is a positive constant, and continuous in $\overline{R}$ satisfying
\begin{align*}
\begin{cases}
|p(z)|\leq \epsilon_1, &\quad z\in (\delta, K],\\
|p(z)|\leq M, &\quad z\in R,
\end{cases}
\end{align*}
where $\delta, K, \epsilon_1$ and $M$ are positive constants. Then there exists a function $\mu(z)$ with $z\in (K, +\infty)$ satisfying 
\begin{equation*}
\mu(z) \geq \frac{64ad}{3\pi^2(a^2 + 4d^2)} e^{\frac{\pi}{2d}(\frac{a}{2} - z)},
\end{equation*}
where $a = K - \delta$, such that
\begin{align*}
|p(z)|\leq M\epsilon^{\mu(z)}\quad \forall\, z\in (K, +\infty).
\end{align*}
\end{lemma}

\begin{lemma}\label{ac_1}
Let $f$ be a real-valued function and $\|f\|_{L^2(B_R)}\leq Q$.
Then there exist positive constants $d$ and $\delta, K$ satisfying
$0< \delta<K$, which do not depend on $f$ and $Q$, such that 
\[
|I(k)| \lesssim
Q^2e^{4R(\sigma + 2)\kappa}\epsilon_1^{2\mu(k)} \quad \forall\,
k\in (K, +\infty)
\]
and
\begin{align*}
\epsilon_1^2 =  \int_\delta^{K} \int_{\partial B_R} \|u(x, k)\|^2_{\partial B_R}{\rm d}s(x){\rm d}k,\quad \mu(k) \geq \frac{64ad}{3\pi^2(a^2 + 4d^2)} e^{\frac{\pi}{2d}(\frac{a}{2} -
k)},
\end{align*}
where $a = K - \delta.$
\end{lemma}

\begin{proof}
Let
\begin{align*}
I_1(k) &= \int_\delta^k \int_{\partial B_R}\Big( ( \omega^4 + \omega^2)u(x, \omega)u(x, -\omega) + \omega^2\nabla u(x, \omega) \cdot \nabla u(x, -\omega) \\
&\quad+ (\omega^2 + 1) \Delta u(x, \omega)\Delta u(x, -\omega) + \nabla\Delta u(x, \omega)\cdot  \nabla\Delta u(x, -\omega)\Big){\rm d}s(x){\rm d}\omega,
\end{align*}
where $k\in\mathcal{R}$. Following similar arguments as those in the proof of Lemma \ref{direct}, we may show that $R(-k)$ is also analytic for $k\in\mathcal{R}$. Since $f$ is real-valued, we have $ \overline{u(x, k)} = u(x, -k)$ for $k\in\mathbb R$, which gives
\[
I_1(k) = I(k), \quad k>0.
\]
It follows from Lemma \ref{direct} that 
\begin{align*}
|I_1(k)|\lesssim Q^2e^{C\sigma^2}e^{4R(\sigma+1)|k|}, \quad k\in\mathcal{R},
\end{align*}
which gives 
\begin{align*}
e^{-4R(\sigma+2)|k|}|I_1(k)|\lesssim Q^2 e^{C\sigma^2}, \quad k\in\mathcal{R}.
\end{align*}
An application of Lemma \ref{ac} leads to 
\[
\big| e^{-4R(\sigma+2)|k|} I(k)\big| \lesssim
Q^2\epsilon^{2\mu(k)} \quad \forall\, k\in (K, +\infty),
\]
where 
\[
\mu(k) \geq \frac{64ad}{3\pi^2(a^2 + 4d^2)} e^{\frac{\pi}{2d}(\frac{a}{2} - k)},
\]
which completes the proof. 
\end{proof}

Here we state a simple uniqueness result for the inverse source problem. 

\begin{theorem}
Let $f\in L^2(B_R)$ and $I\subset\mathbb R^+$ be an open interval. Then the source function $f$ can be uniquely determined by the multi-frequency Cauchy data $\{u(x, k), \nabla u(x, k), \Delta u(x, k), \nabla \Delta u(x, k): x\in\partial B_R, k\in I\}$.
\end{theorem}

\begin{proof}
Let $u(x, k) = \nabla u(x, k) =\Delta u(x, k) = \nabla\Delta u(x, k) = 0$ for all $x\in\partial B_R$ and $k\in I$.  It suffices to prove that $f (x) = 0$. By Lemma \ref{direct}, $u(x, k)$ is analytic in the infinite slab $\mathcal{R}$ for any $\delta>0$, which implies that $u(x, k) = \Delta u(x, k) = 0$ for all $k\in\mathbb R^+$.  We conclude from Lemma \ref{control} that $f = 0.$
\end{proof}

The following result concerns the estimate of $u(x, k)$ for high wavenumbers. 

\begin{lemma}\label{tail}
Let $f\in H^{n}(B_R)$ and $\|f\|_{H^{n}(B_R)}\leq Q$. Then the following estimate holds: 
\begin{align*}
&\int_s^\infty \|u(x, k)\|^2_{\partial B_R}{\rm d}k \lesssim \frac{1}{s^{n-3}}\|f\|^2_{H^n(B_R)}.
\end{align*}
\end{lemma}

\begin{proof}
Recall the identity
\begin{align}\label{tail_s1}
\int_s^\infty \|u(x, k)\|^2_{\partial B_R}{\rm d}k &= \int_s^\infty \int_{\partial B_R} \Big(( k^4 + k^2)|u(x, k)|^2 + k^2|\nabla u(x, k)|^2 \notag\\
&\quad + (k^2 + 1)|\Delta u(x, k)|^2 + |\nabla\Delta u(x, k)|^2 \Big){\rm d}s(x){\rm d}k.
\end{align}
Using the decomposition
\[
R(\kappa) = (\Delta^2 - \kappa^4)^{-1} = \frac{1}{2\kappa^2} \big[(-\Delta - \kappa^2)^{-1} - (-\Delta + \kappa^2)^{-1}\big],
\]
we obtain 
\begin{align*}
u(x) = \int_{B_R}\frac{1}{2\kappa^2} \Big(\frac{e^{{\rm i}\kappa |x - y|}}{4\pi|x - y|} - \frac{e^{-\kappa |x - y|}}{4\pi|x - y|}\Big) f(y){\rm d}y,\quad x\in\partial B_R.
\end{align*}

For instance, we consider one of the integrals on the right-hand side of \eqref{tail_s1}
\begin{align*}
J :&= \int_s^\infty k^4 |u(x, k)| {\rm d}k\\
&= \int_s^\infty k^4 \Big|\int_{B_R}\frac{1}{2\kappa^2} \Big(\frac{e^{{\rm i}\kappa |x - y|}}{4\pi|x - y|} - \frac{e^{-\kappa |x - y|}}{4\pi|x - y|}\Big) f(y){\rm d}y \Big|^2 {\rm d}k.
\end{align*}
Using the spherical coordinates $r = |x - y|$ originated at $y$, we have
\begin{align*}
J = \frac{1}{8\pi} \int_s^\infty \int_{\partial B_R} k^2 \Big|\int_0^{2\pi} {\rm d}\theta \int_0^\pi \sin\varphi {\rm d}\varphi \int_0^\infty 
(e^{{\rm i}\kappa r} - e^{-\kappa r}) fr {\rm d}r \Big|^2{\rm d}s(x){\rm d}k.
\end{align*}
By the integration by parts and noting $x\in\partial B_R$ and $\text{supp}f \subset B_{\hat{R}} \subset B_R$ for some $\hat{R}<R$, we obtain 
\begin{align*}
J = \frac{1}{4\pi} \int_s^\infty \int_{\partial B_R} k^2 \Big|\int_0^{2\pi} {\rm d}\theta \int_0^\pi \sin\varphi {\rm d}\varphi \int_{R - \hat{R}}^{2R} 
\Big(\frac{e^{{\rm i}\kappa r}}{({\rm i}\kappa)^n} - \frac{e^{-\kappa r}}{(-\kappa)^n} \Big) \frac{\partial^n (fr)}{\partial r^n} {\rm d}r \Big|^2{\rm d}s(x){\rm d}k.
\end{align*}
Since $x\in\partial B_R$ and $|\kappa| \geq k^{1/2}$ for $k>0$, we get from direction calculations that 
\begin{align*}
J\lesssim \|f\|^2_{H^n(B_R)}\int_s^\infty k^{2-n} {\rm d}k \lesssim \frac{1}{s^{n-3}}\|f\|^2_{H^n(B_R)}.
\end{align*}
The other integrals on the right-hand side of \eqref{tail_s1} can be estimated similarly. The details are omitted for brevity. 
\end{proof}

Define a real-valued function space
\[
\mathcal C_Q = \{f \in H^{n}(B_R):  n\geq 4, \|f\|_{H^{n}(B_R)}\leq Q, ~ \text{supp}
f\subset B_{\hat{R}}\subset B_R, ~ f: B_R \rightarrow \mathbb R \},
\]
where $\hat{R}<R$. Now we are in the position to present the main result of this paper.

\begin{theorem}
Let $u(x,\kappa)$ be the solution of the scattering problem \eqref{eqn} corresponding to the source $f\in \mathcal C_Q$.
Then for $\epsilon$ sufficiently small, the following estimate holds:
\begin{align}\label{stability}
\|f\|_{L^2( B_R)}^2  \lesssim e^{C\sigma^2}\Big(\epsilon^2 +
\frac{Q^2}{K^{\frac{1}{2}(n-3)}(\ln|\ln\epsilon|)^{\frac{1}{2}(n-3)}}\Big),
\end{align}
where
\begin{align*}
\epsilon:= \int_0^K \|u(x, k)\|^2_{\partial B_R} {\rm d}k = \int_0^\delta \|u(x, k)\|^2_{\partial B_R} {\rm d}k + \epsilon_1^2.
\end{align*}
\end{theorem}

\begin{proof}
We can assume that $\epsilon \leq e^{-1}$, otherwise the estimate is obvious. 

First, we link the data $I(k)$ for large wavenumber $k$ satisfying $k\leq L$ with the given data $\epsilon_1$ of small wavenumber by using the analytic continuation in Lemma \ref{ac_1}, where $L$ is some large positive integer to be determined later. It follows from Lemma \ref{ac_1} that 
\begin{align*}
I(k) & \lesssim Q^2e^{c|\kappa|} \epsilon_1^{\mu(\kappa)}\\
& \lesssim Q^2{\rm exp}\{c\kappa - \frac{c_2a}{a^2 + c_3}e^{c_1(\frac{a}{2} - \kappa)}
|{\ln}\epsilon_1|\}\\
& \lesssim  Q^2{\rm exp} \{  - \frac{c_2a}{a^2 + c_3}e^{c_1(\frac{a}{2} - \kappa)}|{\ln}\epsilon_1| (1 -  \frac{c_4\kappa(a^2 + c_3)}{a} e^{c_1(\kappa - \frac{a}{2})}|{\ln}\epsilon_1|^{-1})\}\\
& \lesssim Q^2{\rm exp} \{  - \frac{c_2a}{a^2 + c_3}e^{c_1(\frac{a}{2} - L)}|{\ln}\epsilon_1| (1 -  \frac{c_4L(a^2 + c_3)}{a} e^{c_1(L - \frac{a}{2})}|{\ln}\epsilon_1|^{-1})\}\\
&\lesssim Q^2{\rm exp} \{  - b_0e^{-c_1L}|{\ln}\epsilon_1| (1 - b_1L e^{c_1L }|{\ln}\epsilon_1|^{-1})\},
\end{align*}
where $c, c_i, i=1,2$ and $b_0, b_1$ are constants. Let
\begin{align*}
L = 
\begin{cases}
\left[\frac{1}{2c_1}\ln|\ln \epsilon_1|\right], &\quad k\leq \frac{1}{2c_1} \ln|\ln\epsilon_1|,\\
k, &\quad k>  \frac{1}{2c_1}\ln|\ln\epsilon_1|.
\end{cases}
\end{align*}

If $K\leq  \frac{1}{2c_1}\ln|\ln\epsilon_1|$, we obtain for sufficiently small $\epsilon_1$ that 
\begin{align*}
I(k)&\lesssim Q^2{\rm exp} \{  - b_0e^{-c_1L}|{\ln}\epsilon_1| (1 - b_1L e^{c_1L }|{\ln}\epsilon_1|^{-1})\}\\
& \lesssim Q^2\exp\{-\frac{1}{2}b_0e^{-c_1L}|\ln \epsilon_1|\}.
\end{align*}
Noting $e^{-x}\leq \frac{(2n+3)!}{x^{2n+3}}$ for $x>0$, we have 
\begin{align*}
I(L) \lesssim Q^2
e^{(2n+3)c_1L}|\ln\epsilon_1|^{-(2n+3)}.
\end{align*}
Taking $L=\frac{1}{2c_1}\ln|\ln\epsilon_1|$, combining the above estimates, Lemma \ref{control}
and  Lemma \ref{tail}, we get
\begin{align*}
\|f\|^2_{L^2(B_R)}&\lesssim e^{C\sigma^2}\Big(\epsilon^2 + I(L)
+ \int_L^\infty \int_{\partial B_R} \|u(x, k)\|^2_{\partial B_R}{\rm d}k\Big)\\
&\lesssim e^{C\sigma^2}\Big(\epsilon^2 + Q^2e^{(2n+3)c_1L}|\ln\epsilon_1|^{-(2n+3)}+ \frac{Q^2}{L^{n-3}}\Big)\\
&\lesssim e^{C\sigma^2}\Big(\epsilon^2 + Q^2\left(|\ln\epsilon_1|^{\frac{2n+3}{2}}|\ln\epsilon_1|^{-(2n+3)}+(\ln|\ln\epsilon_1|)^{3-n}\right)\Big)\\
&\lesssim e^{C\sigma^2}\Big(\epsilon^2 + Q^2\left(|\ln\epsilon_1|^{-\frac{2n+3}{2}}+(\ln|\ln\epsilon_1|)^{3-n}\right)\Big)\\
&\lesssim e^{C\sigma^2}\Big(\epsilon^2 + Q^2(\ln|\ln\epsilon_1|)^{3-n}\Big)\\
&\lesssim e^{C\sigma^2}\Big(\epsilon^2 +
\frac{Q^2}{K^{\frac{1}{2}(n-3)}(\ln|\ln\epsilon_1|)^{\frac{1}{2}(n-3)}}\Big)\\
&\lesssim e^{C\sigma^2}\Big(\epsilon^2 +
\frac{Q^2}{K^{\frac{1}{2}(n-3)}(\ln|\ln\epsilon|)^{\frac{1}{2}(n-3)}}\Big), 
\end{align*}
where we have used $|\ln\epsilon_1|^{1/2}\geq \ln|\ln\epsilon_1|$ for
sufficiently small $\epsilon_1$ and $\ln|\ln\epsilon_1| \geq \ln|\ln\epsilon|$.

If $K > \frac{1}{2c_1}\ln|\ln\epsilon_1|$, we have from Lemma \ref{tail} that 
\begin{align*}
\|f\|_{L^2( B_R)}^2 &\lesssim e^{C\sigma^2}\Big(\epsilon^2 + \int_{K}^\infty \int_{\partial B_R}\|u(x, k)\|^2_{\partial B_R}{\rm d}k\Big)\\
&\lesssim e^{C\sigma^2}\Big(\epsilon^2 +
\frac{Q^2}{K^{n-3}}\Big)\\
&\lesssim e^{C\sigma^2}\Big(\epsilon^2 +
\frac{Q^2}{K^{\frac{1}{2}(n-3)}(\ln|\ln\epsilon|)^{\frac{1}{2}(n-3)}}\Big),
\end{align*}
which completes the proof.
\end{proof}

It can be observed that for a fixed damping coefficient $\sigma$, the stability \eqref{stability} consists of two parts: the data discrepancy and the high frequency tail. The former is of the Lipschitz type. The latter decreases as $K$ increases which makes the problem have an almost Lipschitz stability. But the stability deteriorates exponentially as the damping coefficient $\sigma$ increases. 

\begin{appendix}

\section{An exact observability bound}\label{exact}

Consider the initial value problem for the damped biharmonic plate wave equation 
\begin{align}\label{pe}
\begin{cases}
\partial_t^2 U(x, t) + \Delta^2 U(x, t) + \sigma\partial_t U(x, t) = 0, &\quad (x, t)\in B_R\times (0, +\infty),\\
U(x, 0) = 0, \quad \partial_tU(x, 0) = f(x), &\quad x\in B_R.
\end{cases}
\end{align}
The following theorem presents an exact observability bound for the above equation. The proof follows closely from that in \cite[Theorem 3.1]{IL}.

\begin{theorem}\label{control_1}
Let the observation time $4(2R + 1)< T < 5(2R + 1)$. Then there exists a constant $C$ depending on the domain $B_R$ such that
\begin{align}\label{bc}
\|f\|^2_{L^2(B_R)} &\leq Ce^{C\sigma^2} \big( \|\partial^2_t U\|^2_{L^2(\partial B_R\times (0, T))}+ \|\partial_t U\|^2_{L^2(\partial B_R\times (0, T))} + \|\partial_t \nabla U\|^2_{L^2(\partial B_R\times (0, T))} \notag\\
&\quad +  \|\partial_t \Delta U\|^2_{L^2(\partial B_R\times (0, T))} + \|\Delta U\|^2_{L^2(\partial B_R\times (0, T))} + \|\nabla\Delta U\|^2_{L^2(\partial B_R\times (0, T))}\big).
\end{align}
\end{theorem}

Before showing the proof, we introduce the energies
\begin{align*}
E(t) &= \frac{1}{2}\int_{\Omega} \big(|\partial_t U(x, t)|^2 + |\Delta U(x, t)|^2 + |U(x, t)|^2\big) {\rm d}x,\\
E_0(t) &= \frac{1}{2}\int_{\Omega} \big(|\partial_t U(x, t)|^2 + |\Delta U(x, t)|^2 \big) {\rm d}x,
\end{align*}
and denote
\begin{align*}
F^2 &= \int_{\partial \Omega\times (t_1, t_2)}\big( |\partial^2_t U(x, t)|^2 + |\partial_t U(x, t)|^2 +  |\partial_t \nabla U(x, t)|^2 \\
&\quad + |\partial_t\Delta U(x, t)|^2 + |\Delta U(x, t)|^2 + |\nabla \Delta U(x, t)|^2 \big){\rm d}s(x){\rm d}t.
\end{align*}

\begin{lemma}
Let $U$ be a solution of the damped biharmonic plate wave equation \eqref{pe} with the initial value $f\in H^1(B_R)$, $supp f\subset B_R$. Let $0\leq t_1<t_2\leq T$ and $1\leq 2\sigma$. Then the following estimates holds: 
\begin{align}\label{energy_1}
E(t_2) &\leq e^{4(t_2 - t_1)^2}(2E(t_1) + F^2), \\
\label{energy_2}
E(t_2)&\leq e^{(2\sigma + 4(t_2 - t_1))(t_2 - t_1)}(E(t_2) + F^2).
\end{align}
\end{lemma}

\begin{proof}

Multiplying both sides of \eqref{pe} by $(\partial_t U) e^{\theta t}$ and integrating over $\Omega\times (t_1, t_2)$ give
\[
\int_{\Omega\times ((t_1, t_2)} \Big(\frac{1}{2}\partial_t (\partial_t U)^2 + \Delta^2U\partial_t U + \sigma(\partial_t U)^2\Big)e^{\theta t}\,{\rm d}x\,{\rm d}t = 0.
\]
Using the integration $\Delta^2U\partial_t U$ by parts over $\Omega$ and noting $\Delta U \partial_t(\Delta U) = \frac{1}{2}\partial_t|\Delta U|^2$, we obtain
\begin{align*}
&\int_{t_1}^{t_2} (\partial_t E_0(t))e^{\theta t}{\rm d}t + \int_{\Omega\times(t_1, t_2)} \sigma(\partial_t U)^2e^{\theta t}{\rm d}x{\rm d}t\\
&\quad + \int_{\partial\Omega\times(t_1, t_2)} (\partial_\nu(\Delta U)\partial_tU - \Delta U \partial_t(\partial_\nu U))e^{\theta t}{\rm d}s(x){\rm d}t = 0.
\end{align*}
Hence, 
\begin{align*}
E_0(t_2)e^{\theta t_2} - E_0(t_1)e^{\theta t_1} &= \int_{\Omega\times(t_1, t_2)} \Big(\frac{\theta}{2} ((\partial_t U)^2 + |\Delta U|^2) - \sigma(\partial_t U)^2\Big)e^{\theta t}{\rm d}x{\rm d}t\\
&\quad - \int_{\partial\Omega\times(t_1, t_2)} (\partial_\nu(\Delta U)\partial_tU - \Delta U \partial_t(\partial_\nu U))e^{\theta t}{\rm d}s(x){\rm d}t = 0.
\end{align*}
Letting $\theta = 0$, using Schwartz's inequality, and noting $\sigma>0$, we get
\begin{align*}
E_0(t_2)&\leq E_0(t_1) + \int_{\Omega\times(t_1, t_2)} (-\sigma)(\partial_t U)^2{\rm d}x{\rm d}t\\
&\quad + \frac{1}{2}\int_{\partial\Omega\times(t_1, t_2)} \Big((\partial_t U)^2 + (\partial_t(\partial_\nu U))^2\Big){\rm d}s(x){\rm d}t\\
&\quad + \frac{1}{2}\int_{\partial\Omega\times(t_1, t_2)} \Big((\Delta U)^2 + (\partial_\nu(\Delta U))^2\Big){\rm d}s(x){\rm d}t\\
&\leq E_0(t_1) + F^2.
\end{align*}
Similarly, letting $\theta = 2\sigma$, we derive 
\begin{align*}
E_0(t_1)e^{2\sigma t_1}&\leq E_0(t_2)e^{2\sigma t_2} + \int_{\Omega\times(t_2, t_1)}-\sigma (\Delta U)^2{\rm d}x{\rm d}t\\
&\quad + \frac{1}{2}\int_{\partial\Omega\times(t_1, t_2)} \Big((\partial_t U)^2 + (\partial_t(\partial_\nu U))^2\Big)e^{2\sigma t}{\rm d}s(x){\rm d}t\\
&\quad + \frac{1}{2}\int_{\partial\Omega\times(t_1, t_2)} \Big((\Delta U)^2 + (\partial_\nu(\Delta U))^2\Big)e^{2\sigma t}{\rm d}s(x){\rm d}t\\
&\leq E_0(t_2)e^{2\sigma t_2} + \frac{1}{2}\int_{\partial\Omega\times(t_1, t_2)} \Big((\partial_t U)^2 + (\partial_t(\partial_\nu U))^2\Big)e^{2\sigma t}{\rm d}s(x){\rm d}t\\
&\quad + \frac{1}{2}\int_{\partial\Omega\times(t_1, t_2)} \Big((\Delta U)^2 + (\partial_\nu(\Delta U))^2\Big)e^{2\sigma t}{\rm d}s(x){\rm d}t.
\end{align*}
which gives 
\begin{align*}
E_0(t_1) \leq e^{2\sigma(t_2 - t_1)} (E_0(t_2) + F^2).
\end{align*}
The proof is completed by following similar arguments as those in \cite[Lemma 3.2]{IL}. 
\end{proof}

Now we return to the proof of Theorem \ref{bc}

\begin{proof}[Proof of Theorem \ref{bc}]
Let $\varphi(x, t) = |x - a|^2 - \theta^2(t - \frac{T}{2})^2$, where ${\rm dist} (a, \Omega) = 1, \theta = \frac{1}{2}$. 
Using the Carleman-type estimate in \cite{Yuan}, we obtain 
\begin{align}\label{carleman}
&\tau^6 \int_Q |U|^2e^{2\tau\varphi}{\rm d}x{\rm d}t + 
\tau^3 \int_Q |\partial_tU|^2e^{2\tau\varphi}{\rm d}x{\rm d}t + \tau \int_Q |\Delta U|^2e^{2\tau\varphi}{\rm d}x{\rm d}t\notag\\
&\lesssim\int_Q ((\partial_t^2 + \Delta^2)U)^2e^{2\tau\varphi}{\rm d}x{\rm d}t\notag\\
&\quad + \int_{\partial Q}\tau^6(|\partial_\nu \Delta U|^2 +  |\partial_t \Delta U|^2 + |\partial_t^2 U)|^2)e^{2\tau\varphi}{\rm d}s(x){\rm d}t.
\end{align}
It is easy to see that $1 - \theta^2\varepsilon_0^2\leq\varphi$ on $\Omega\times \{|t - \frac{T}{2}|<\varepsilon_0\}$ for some positive $\varepsilon<1$. Then we have from \ref{energy_2} that 
\begin{align}\label{lhs}
&\tau^6 \int_Q |U|^2e^{2\tau\varphi}{\rm d}x{\rm d}t + 
\tau^3 \int_{Q} |\partial_tU|^2e^{2\tau\varphi}{\rm d}x{\rm d}t + \tau \int_Q |\Delta U|^2e^{2\tau\varphi}{\rm d}x{\rm d}t \notag\\
&\geq \tau^6 \int_{\Omega\times(\frac{T}{2} - \varepsilon_0, \frac{T}{2} + \varepsilon_0)} |U|^2e^{2\tau(1 - \theta^2\varepsilon_0^2)}{\rm d}x{\rm d}t + 
\tau^3 \int_{\Omega\times(\frac{T}{2} - \varepsilon_0, \frac{T}{2} + \varepsilon_0)} |\partial_tU|^2e^{2\tau(1 - \theta^2\varepsilon_0^2)}{\rm d}x{\rm d}t \notag\\
&\quad + \tau \int_{\Omega\times(\frac{T}{2} - \varepsilon_0, \frac{T}{2} + \varepsilon_0)} |\Delta U|^2e^{2\tau(1 - \theta^2\varepsilon_0^2)}{\rm d}x{\rm d}t \notag\\
&\geq \tau e^{2\tau(1 - \theta^2\varepsilon_0^2)} \int_{\Omega\times(\frac{T}{2} - \varepsilon_0, \frac{T}{2} + \varepsilon_0)} E(t) {\rm d}t \notag\\
&\geq \tau e^{2\tau(1 - \theta^2\varepsilon_0^2)}\varepsilon_0 (2e^{-(2\sigma + 4T)T} E(0) - F^2). 
\end{align}
Moreover, it follows from \eqref{energy_2} and $\varphi\leq (2R + 1)^2 - \theta^2T^2/4$ on $\Omega\times (0, T)$ that 
\begin{align*}
&\tau^6 \int_Q |U|^2e^{2\tau\varphi}{\rm d}x{\rm d}t + 
\tau^3 \int_Q |\partial_tU|^2e^{2\tau\varphi}{\rm d}x{\rm d}t + \tau \int_Q |\Delta U|^2e^{2\tau\varphi}{\rm d}x{\rm d}t\\
&\leq \tau^6 e^{2\tau((2R + 1)^2 - \theta^2T^2/4)}(E(0) + E(T))\\
&\leq \tau^6 e^{2\tau((2R + 1)^2 - \theta^2T^2/4)} ((e^{4T^2} + 1)E(0) + e^{4T^2}F^2). 
\end{align*}
By \eqref{carleman} and \eqref{lhs}, we obtain
\begin{align}\label{combine}
&\tau e^{2\tau(1 - \theta^2\varepsilon_0^2)}\varepsilon_0 e^{-(2\sigma + 1 + 4T)T} E(0)\notag\\
&\quad +\tau^6 \int_Q |U|^2e^{2\tau\varphi}{\rm d}x{\rm d}t + 
\tau^3 \int_Q |\partial_tU|^2e^{2\tau\varphi}{\rm d}x{\rm d}t + \tau \int_Q |\Delta U|^2e^{2\tau\varphi}{\rm d}x{\rm d}t\notag\\
&\leq \Big(\sigma^2\int_Q |\partial_t U|^2e^{2\tau\varphi}{\rm d}x{\rm d}t + \int_{\partial Q}\tau^6(|\partial_\nu \Delta U|^2 +  |\partial_t \Delta U|^2 + |\partial_t^2 U)|^2)e^{2\tau\varphi}{\rm d}s(x){\rm d}t\notag\\
&\quad + (\tau e^{2\tau(1 - \theta^2\varepsilon_0^2)} + \tau^6 e^{2\tau((2R + 1)^2 - \theta^2T^2/4)}e^{4T^2}) F^2 +  \tau^6 e^{2\tau((2R + 1)^2 - \theta^2T^2/4)}e^{4T^2} E(0)\Big).
\end{align}
Choosing $\tau$ sufficiently large, we may remove the first integral on the right hand side of \eqref{combine}. We also choose $T^2 = 4\frac{(2R+1)^2}{\theta^2} + 4\varepsilon_0^2$ and $\tau = (2\sigma + 8T)T + \ln(2(\varepsilon_0)^{-1}C) + C\sigma^2$. Noting $\tau^5e^{-\tau}\leq 5!$, we have
\begin{align*}
\tau^5 e^{2\tau((2R + 1)^2 - \theta^2T^2/4 - 1 + \theta^2\varepsilon_0^2) + (2\sigma + 8T)T} &= \tau^5 e^{-2\tau + (2\sigma + 8T)T}\\
&\leq 5! e^{-\tau + (2\sigma + 8T)T}\leq\frac{\varepsilon_0}{2C}.
\end{align*}
In addition, since $T\leq 5(2R + 1)$, it follows that 
\begin{align*}
\tau^5 e^{2\tau((2R + 1)^2 - 1 + \theta^2\varepsilon_0^2 + (2\sigma + 4T)T)} \leq \tau^5 e^{2((2\sigma + 8T)T + C\sigma^2 + C)(2R + 1)^2 + (2\sigma + 4T)T}\leq Ce^{C\sigma^2}. 
\end{align*}
 Using the above inequality and the inequality $\varphi<(2R + 1)^2$ on $Q$ and dividing both sides in \eqref{combine} by the factor of $E(0)$ on the left hand side, we obtain
\[
E(0)\leq Ce^{C\sigma^2} F^2.
\]
Since $f$ is supported in $\Omega$, there holds $\|U\|_{L^2(\partial B_R\times (0, T))}\leq C\|\partial_t U\|_{L^2(\partial B_R\times (0, T))}$, which completes the proof.
\end{proof}

\section{A decay estimate}\label{time decay}

We prove a decay estimate for the solution of the initial value problem of the damped plate wave equation
\begin{align}\label{Kirchhoff_2}
\begin{cases}
\partial_t^2 U(x, t) + \Delta^2 U(x, t) + \sigma\partial_t U(x, t) = 0, &\quad (x, t)\in \mathbb R^3\times (0, +\infty),\\
U(x, 0) = 0, \quad \partial_tU(x, 0) = f(x), &\quad x\in\mathbb R^3,
\end{cases}
\end{align}
where $f(x)\in L^1(\mathbb R^3)\cap H^s(\mathbb R^3)$. By the Fourier transform, the solution $U(x, t)$ of \eqref{Kirchhoff_2} is given as 
\begin{align*}
U(x, t) = \mathcal{F}^{-1} (m_\sigma (t, \xi) \hat{f}(\xi))(x),
\end{align*}
where $\mathcal{F}^{-1}$ denotes the inverse Fourier transform, 
\begin{align*}
m_\sigma(t, \xi) = \frac{e^{-\frac{\sigma}{2}t}}{\sqrt{\sigma^2 - 4|\xi|^4}} \Big(e^{\frac{1}{2}t \sqrt{\sigma^2 - 4|\xi|^4}} - e^{-\frac{1}{2}t \sqrt{\sigma^2 - 4|\xi|^4}}\Big), 
\end{align*}
and $\hat{f}(\xi)$ is the Fourier transform of $f$, i.e., 
\[
\hat{f}(\xi) = \frac{1}{(2\pi)^3} \int_{\mathbb R^3} e^{-{\rm i}x\cdot\xi} f(x) {\rm d}x. 
\]

Let $\sqrt{\sigma^2 - 4|\xi|^4} = {\rm i} \sqrt{4|\xi|^4 - \sigma^2}$ when $|\xi|^4>\frac{\sigma^2}{4}$. Then we have 
\begin{align*}
m_\sigma(t, \xi) = 
\begin{cases}
e^{-\frac{\sigma}{2}t} \frac{\sinh (\frac{t}{2}\sqrt{\sigma^2 - 4|\xi|^4 })}{\sqrt{\sigma^2 - 4|\xi|^4}},&\quad |\xi|^4<\frac{\sigma^2}{4},\\
e^{-\frac{\sigma}{2}t} \frac{\sin (\frac{t}{2}\sqrt{4|\xi|^4 - \sigma^2})}{\sqrt{4|\xi|^4 - \sigma^2}},&\quad |\xi|^4>\frac{\sigma^2}{4}.
\end{cases}
\end{align*}
It is clear to note from the representation of $m_\sigma(t, \xi)$ that the solution $U(x, t)$ depends on both of the low and high frequency of $\xi$. In fact, the solution $U(x, t)$ behaves as a ``parabolic type" of $e^{-t\Delta^2} f$ for the low frequency part, while for the high frequency part it behaves like a ``dispersive type" of $e^{{\rm i}t\Delta^2} f$.

\begin{theorem}\label{thm-decay}
Let $U(x, t)$ be the solution of \eqref{Kirchhoff_2}. Then $U(x, t)$ satisfies the decay estimate
\begin{align}\label{de}
{\rm sup}_{x\in\mathbb R^3} |\partial^\alpha_x\partial_t^j U(x, t)| \lesssim (1 + t)^{-\frac{3 + |\alpha|}{4}} \|f\|_{L^1(\mathbb R^3)} + e^{-ct} \|f\|_{H^s(\mathbb R^3)},
\end{align}
where $j\in\mathbb N$, $\alpha$ is a multi-index vector in $\mathbb N^3$ such that $\partial_\alpha = \partial_{x_1}^{\alpha_1}\,\partial_{x_2}^{\alpha_2}\,\partial_{x_3}^{\alpha_3}$, $s>2j + |\alpha| - \frac{1}{2}$ and $c>0$ is some positive constant. In particular, for $|\alpha| = s = 0$, the following estimate holds:
\begin{align}\label{FD}
{\rm sup}_{x\in\mathbb R^3} |U(x, t)|\lesssim (1 + t)^{-\frac{3}{4}} (\|f\|_{L^1(\mathbb R^3)} + \|f\|_{L^2(\mathbb R^3)}).
\end{align}
\end{theorem}

\begin{remark}\label{plancherel}
The estimate \eqref{FD} provides a time decay of the order $O((1 + t)^{-\frac{3}{4}})$ for $U(x, t)$ uniformly for all $x\in\mathbb R^3$, which gives 
\begin{align*}
\sup_{x\in\mathbb R^3} \int_0^\infty |U(x, t)|^2 {\rm d}t \lesssim \int_0^\infty (1 + t)^{-3/2}{\rm d}t < +\infty.
\end{align*}
Hence, let $U(x, t) = 0$ when $t<0$, then $U(x, t)$ has a Fourier transform $\hat{U}(x, k) \in L^2(\mathbb R)$ for each $x\in\mathbb R^3$. Moreover, the following Plancherel equality holds: 
\begin{align*}
\int_0^{+\infty} |U(x, t)|^2{\rm d}t = \int_{-\infty}^{+\infty} |\hat{U}(x, k)|^2 {\rm d}k.
\end{align*} 
\end{remark}

\begin{remark}\label{rd}
To study the inverse source problem, it suffices to assume that $f\in H^4(\mathbb R^3)$. In this case,  it follows from the above theorem that both $\partial_t^2U(x, t)$ and $\Delta^2 U(x, t)$ are continuous functions. Moreover, we have from \eqref{de} that the following estimate holds: 
\begin{align*}
{\rm sup}_{x\in\mathbb R^3} |\partial_t^j U(x, t)|& \lesssim (1 + t)^{-\frac{3}{4}} \|f\|_{L^1(\mathbb R^3)} + e^{-ct} \|f\|_{H^s(\mathbb R^3)},\quad j=1, 2,\\
{\rm sup}_{x\in\mathbb R^3} |\partial^\alpha_xU(x, t)|& \lesssim (1 + t)^{-\frac{3 + |\alpha|}{4}} \|f\|_{L^1(\mathbb R^3)} + e^{-ct} \|f\|_{H^s(\mathbb R^3)},\quad |\alpha|\leq 4.
\end{align*}
\end{remark}

\begin{proof}
Without loss of generality, we may assume that $\sigma = 1$, and then
\begin{align*}
m_\sigma(t, \xi) = \frac{e^{-\frac{1}{2}t}}{\sqrt{1 - 4|\xi|^4}} \Big(e^{\frac{1}{2}t \sqrt{1 - 4|\xi|^4}} - e^{-\frac{1}{2}t \sqrt{1 - 4|\xi|^4}}\Big).
\end{align*}

First we prove \eqref{de} for $j = 0$. Choose $\chi\in C_0^\infty(\mathbb R^3)$ such that ${\rm supp}\chi\subset B(0, \frac{1}{2})$ and $\chi(\xi) = 1$ for $|\xi|\leq\frac{1}{4}$. Let 
\begin{align*}
U(x, t) &= \mathcal{F}^{-1} (m(t, \xi)\chi(\xi)\hat{f})  +  \mathcal{F}^{-1}(m(t, \xi)(1 - \chi(\xi))\hat{f})\\
&:= U_1(x, t) + U_2(x, t).
\end{align*}
For $U_1(x, t)$, since $\sqrt{1 - 4|\xi|^4}\leq 1 - 2|\xi|^4$ when $0\leq |\xi| \leq \frac{1}{2}$, we have for $|\xi|\leq\frac{1}{2}$ that
\[
m(t, \xi) = \frac{1}{\sqrt{1 - 4|\xi|^4}} e^{-\frac{t}{2}(1 \pm \sqrt{1 - 4|\xi|^4})}\leq 2e^{-t|\xi|^4}, \quad t\geq 0.
\]
For each $x\in\mathbb R^3$ we have
\[
\partial^\alpha U_1(x, t) = \int_{\mathbb R^3} e^{{\rm i}x\cdot\xi} ({\rm i}\xi)^\alpha m(t, \xi) \chi(\xi) \hat{f}(\xi) {\rm d}\xi, 
\]
which gives 
\begin{align*}
\sup_{x\in\mathbb R^3} |\partial_x^\alpha U_1(x, t)| \leq \int_{|\xi|\leq\frac{1}{2}} |\xi|^\alpha e^{-t|\xi|^4} |\hat{f}(\xi)|{\rm d}\xi\lesssim \|\hat{f}\|_{L^\infty(\mathbb R^3)} \int_{|\xi|\leq\frac{1}{2}} |\xi|^\alpha e^{-t|\xi|^4}{\rm d}\xi.
\end{align*}
Since
\begin{align*}
\int_{|\xi|\leq\frac{1}{2}} |\xi|^\alpha e^{-t|\xi|^4}{\rm d}\xi \leq
\begin{cases}
C, &\quad 0\leq t \leq 1,\\
t^{-\frac{3 + |\alpha|}{4}}, &\quad t\geq 1,
\end{cases}
\end{align*}
and $\|\hat{f}\|_{L^\infty(\mathbb R^3)} \leq \|f\|_{L^1(\mathbb R^3)}$, we obtain 
\begin{align}\label{U1}
\sup_{x\in\mathbb R^3}|\partial_x^\alpha U_1(x, t)| \lesssim (1 + t)^{-\frac{3 + |\alpha|}{4}} |f\|_{L^1(\mathbb R^3)} \quad \forall \alpha\in\mathbb N^3.
\end{align}

To estimate $U_2(x, t)$, noting 
\[
(1 - \Delta)^{\frac{p}{2}} U_2(x, t) = \int_{\mathbb R^3} e^{{\rm i}x\cdot\xi} (1 + |\xi|^2)^{\frac{p}{2}} m(t, \xi) (1 - \chi(\xi)) \hat{f}(\xi) {\rm d}\xi, 
\]
we have from Plancherel's theorem that 
\begin{align}\label{PU2}
\int_{\mathbb R^3} |(1 - \Delta)^{\frac{p}{2}} U_2(x, t)|^2 {\rm d}x = \int_{\mathbb R^3} (1 + |\xi|^2)^p |m(t, \xi) (1 - \chi(\xi)) \hat{f}(\xi)|^2 {\rm d}\xi.
\end{align}
It holds that
\begin{align*}
|m(t, \xi)|\leq 
\begin{cases}
te^{-\frac{t}{2}(1 - \sqrt{1 - 4|\xi|^4})} \Big|\frac{1 - e^{-t\sqrt{1 - 4|\xi|^4}}}{t\sqrt{1 - 4|\xi|^4}}\Big|\lesssim e^{-\frac{t}{8}}, &\quad \frac{1}{2}<|\xi|\leq \frac{\sqrt{2}}{2},\\
\frac{1}{2}e^{-\frac{t}{2}} \frac{\sin \frac{t}{2}\sqrt{4|\xi|^4 - 1}}{\frac{t}{2}\sqrt{4|\xi|^4 - 1}}\lesssim e^{-\frac{t}{8}}, &\quad \frac{\sqrt{2}}{2}<|\xi|\leq 1,\\
\frac{e^{-\frac{t}{2}}}{\sqrt{4|\xi|^4 - 1}} |\sin \frac{t}{2}\sqrt{4|\xi|^4 - 1}|\leq \frac{e^{-\frac{t}{2}}}{\sqrt{4|\xi|^4 - 1}}, &\quad |\xi|>1.
\end{cases}
\end{align*}
Hence, when $|\xi|\geq\frac{1}{2}$ we have 
\[
|(1 + |\xi|^2)m(t, \xi)| \lesssim e^{-\frac{t}{8}}.
\]
It follows from \eqref{PU2}  that
\begin{align*}
\|U_2(x, t)\|^2_{H^p(\mathbb R^3)} &\leq \int_{|\xi|\geq \frac{1}{2}} | (1 + |\xi|^2)^{\frac{p}{2}} m(t, \xi) \hat{f}(\xi) |^2 {\rm d}\xi\\
&\leq e^{-\frac{t}{4}} \int_{\mathbb R^3} | (1 + |\xi|^2)^{-1 + \frac{p}{2}} \hat{f}(\xi) |^2 {\rm d}\xi = e^{-\frac{t}{4}} \|f\|^2_{H^{p-2}(\mathbb R^3)}.
\end{align*}
On the other hand, by Sobolev's theorem, we have for $p>\frac{3}{2}$ that 
\[
\sup_{x\in\mathbb R^3} |U_2(x, t)| \leq \|U_2(\cdot, t)\|_{H^{p}(\mathbb R^3)} \lesssim e^{-\frac{t}{8}}\|f\|_{H^{p-2}(\mathbb R^3)}.
\]
More generally, for any $\alpha\in\mathbb N^3$ it holds that
\[
(1 - \Delta)^{\frac{p}{2}}\partial_x^\alpha U_2(x, t) = \mathcal{F}^{-1} ((1 + |\xi|^2)^{\frac{p}{2}} m(t, \xi) (1 - \chi(\xi))\widehat{\partial^\alpha f}),
\]
which leads to 
\begin{align}\label{U2}
\sup_{x\in\mathbb R^3} |\partial_x^\alpha U_2(x, t)| \lesssim e^{-\frac{t}{8}} \|\partial^\alpha f\|_{H^{p-2}(\mathbb R^3)}\lesssim e^{-\frac{t}{8}} \|f\|_{H^{s}(\mathbb R^3)}.
\end{align}
Here $s = p - 2 + |\alpha|>|\alpha| - \frac{1}{2}$ by choosing $p > \frac{3}{2}$. Combining the estimate \eqref{U1} with \eqref{U2} yields \eqref{de} for $j = 0.$

Next we consider the general case with $j\neq 0$. Noting  
\[
\partial_t^j U(x, t) = \int_{\mathbb R^3} e^{{\rm i}x\cdot\xi} \partial_t^j m(t, \xi) \hat{f}(\xi) {\rm d}\xi,
\]
we obtain from direct calculations that 
\begin{align*}
\partial_t^j m(t, \xi) &= \partial_j \Big(\frac{e^{-\frac{1}{2}t}}{\sqrt{1 - 4|\xi|^4}} \Big(e^{\frac{1}{2}t \sqrt{1 - 4|\xi|^4}} - e^{-\frac{1}{2}t \sqrt{1 - 4|\xi|^4}}\Big)\Big)\\
&= \sum_{l=0}^j 2^{-j} (\sqrt{1 - 4|\xi|^4})^{l-1} e^{-\frac{t}{2}} \Big(e^{\frac{1}{2}t \sqrt{1 - 4|\xi|^4}} + (-1)^{l+1}e^{-\frac{1}{2}t \sqrt{1 - 4|\xi|^4}}\Big)\\
&:= \sum_{l=0}^j m_l(t, \xi).
\end{align*}
Hence we can write $\partial_t^j U(x, t)$ as 
\begin{align}\label{sw}
\partial_t^j U(x, t) = \sum_{l=0}^j \int_{\mathbb R^3} e^{{\rm i}x\cdot\xi} m_l(t, \xi) \hat{f}(\xi){\rm d}\xi:= \sum_{l=0}^j W_l(x, t).
\end{align}
For each $0\leq l \leq j, \, j\neq 0$, using similar arguments for the case $j = 0$ we obtain 
\begin{align}\label{w}
\sup_{x\in\mathbb R^3} |\partial_x^\alpha W_l(x, t)| \leq (1 + t)^{-\frac{3 + |\alpha|}{4}}\|f\|_{L^1(\mathbb R^3)} + e^{-\frac{t}{8}}\|f\|_{H^{s}(\mathbb R^3)}
\end{align}
for $s>2l + |\alpha| - \frac{1}{2}$. Combining \eqref{sw} and \eqref{w}, we obtain the general estimate \eqref{de}.
\end{proof}

\begin{remark}
For the damped biharmonic plate wave equation, besides the decay estimate \eqref{de}, we can deduce other decay estimates of the $L^p$-$L^q$ type and time-space estimates by more sophisticated analysis for the Fourier multiplier $m(t, \xi)$. For example, it can be proved that 
\[
\|U(x, t)\|_{L^q(\mathbb R^3)} \lesssim (1 + t)^{-\frac{3}{4}(\frac{1}{p} - \frac{1}{q})}\|f\|_{L^p(\mathbb R^3)} + e^{-ct}\|f\|_{W^{q, s}(\mathbb R^3)},
\]
where $1<p\leq q<+\infty$ and $s\geq 3 (\frac{1}{q} - \frac{1}{2}) - 2$. We hope to present the proofs of these $L^p$-$L^q$ estimates and their applications elsewhere. 
\end{remark}

\end{appendix}

\section*{Acknowledgement}

We would like to thank Prof. Masahiro Yamamoto for providing the reference \cite{Yuan} on Carleman estimates of the Kirchhoff plate equation. The research of PL is supported in part by the NSF grant DMS-1912704. The research of XY is supported in part by NSFC (No. 11771165). The research of YZ is supported in part by NSFC (No. 12001222).

\end{document}